\documentclass[12pt]{amsart}
\RequirePackage{amsmath,amsthm,amssymb,amsfonts}
\title[Boundary Smoothness]{
Boundary Smoothness of Analytic Functions}
\author{Anthony G. O'Farrell}
\email{admin@maths.nuim.ie}
\address{Mathematics and Statistics, NUI, Maynooth, Co. Kildare, Ireland}
\date{\today. Misprints in published 2013 version corrected.}
\dedicatory{Dedicated to Lawrence Zalcman\\ on the occasion of his $70$th birthday}

\newtheorem{theorem}{Theorem}[section]
\newtheorem{corollary}[theorem]{Corollary}
\newtheorem{lemma}[theorem]{Lemma}

\newcommand{\B}{\mathbb{B}}
\newcommand{\C}{\mathbb{C}}
\newcommand{\N}{\mathbb{N}}

\newcommand{\bdy}{\textup{bdy}}
\newcommand{\clos}{\textup{clos}}

\newcommand{\dist}{\textup{dist}}

\newcommand{\Lip}{\textup{Lip}}
\newcommand{\lip}{\textup{lip}}
\newcommand{\oh}{\textup{o}}
\newcommand{\Oh}{\textup{O}}

\def\sideremark#1{\ifvmode\leavevmode\fi\vadjust{
\vbox to0pt{\hbox to 0pt{\hskip\hsize\hskip1em
\vbox{\hsize1cm\tiny\raggedright\pretolerance10000
\noindent #1\hfill}\hss}\vbox to8pt{\vfil}\vss}}}

\begin{document}
\begin{abstract}
We consider the behaviour of holomorphic functions 
on a bounded open subset of the plane,
satisfying a Lipschitz condition with exponent
$\alpha$, with $0<\alpha<1$, in the vicinity of
an exceptional boundary point where all such functions exhibit some
kind of smoothness. Specifically, we consider the
relation between the abstract idea of a bounded
point derivation on the algebra of
such functions and the classical complex derivative
evaluated as a limit of difference quotients.
We obtain a result which applies, for example,
when the open set admits an interior cone
at the special boundary point.
\end{abstract}

\maketitle

\section{Introduction}
Let $U$ be a nonempty open subset of $\C$, and let $f:U\to\C$
be holomorphic on $U$.  Suppose $0<\alpha<1$, and
$f$ satisfies a H\"older, or Lipschitz condition
with exponent $\alpha$ on $U$: i.e. there exists $\kappa>0$
such that
\begin{equation}\label{E:1}
 | f(z)-f(w)|\le \kappa|z-w|^\alpha,\ \forall z,w\in U.
\end{equation}
  Then $f$ has a unique 
continuous extension to $Y=\clos(U)$.
This extension also satisfies the Lipschitz condition
with exponent $\alpha$ on $Y$, with the same constant $\kappa$.

Let $b$ belong to the boundary of $U$. It may happen
that all such $f$ are in some sense smoother at $b$ than a
typical H\"older-continuous complex-valued function.
That is, the additional assumption of analyticity 
on $U$ may force additional smoothness at $b$.

The strongest possible smoothness that might occur
would be that all such $f$ are actually holomorphic on a
certain neighbourhood of $b$.  In that case, $b$
is usually referred to as a {\em removable singularity
for $\Lip\alpha$ holomorphic functions on $U$}.  This
phenomenon was first investigated by Dolzhenko,
who showed \cite{Dol} that $b$ is a removable singularity
of this type if and only if there exists $r>0$
such that $\B(b,r)\sim U$ has zero $(1+\alpha)$-dimensional
Hausforff measure, where $\B(b,r)$ denotes the closed disk
having center $b$ and radius $r$.  (It seems appropriate
to mention here that author's
interest in removable singularities, and in this whole
area, was first aroused by Larry Zalcman's Monthly paper
\cite{Z}.)

It may also happen that more limited smoothness 
occurs at a boundary point $b$ that is not removable.
In \cite{LO} Lord and the author considered  the notion of a
{\em continuous point derivation}
at a boundary point, and gave a necessary and
sufficient condition for the existence
of such a nonzero derivation.  The concept of continuous point
derivation comes from the theory of commutative
Banach algebras. If $A$ is a commutative Banach algebra
with character space (maximal ideal space)
$M(A)$, and $\phi\in M(A)$, then a continuous point
derivation on $A$ at $\phi$ is a continuous linear
functional $\partial:A\to\C$ such that the Leibniz
rule
$$ \partial(fg) = \phi(f)\partial(g)+
\partial(f)\phi(g) $$
holds for all $f,g\in A$.  In the present case,
we considered the algebra
$A=A_\alpha(U)$ of all holomorphic functions $f$ on $U$
that belong to the \lq\lq little Lipschitz class",
i.e. are not just Lip$\alpha$ functions, but
have the stronger property that
for each $\epsilon>0$ there exists $\delta>0$
such that for all $z,w\in U$,
$$ |z-w|<\delta \implies 
|f(z)-f(w)|\le \epsilon|z-w|^\alpha.$$
The norm on $A$ is given by
$$ \|f\|_A = \sup_U|f| + \kappa(f), $$
where $\kappa(f)$ is the least value
that works in the inequality \eqref{E:1}. 
The characters on this $A$ are precisely the
evaluations $f\mapsto f(a)$, for $a\in Y$
(as follows from a result of Sherbert \cite{S}),
and we identify the point $a$ with the corresponding character.
At points $a\in U$, the map $\partial: f\mapsto
f'(a)$ is a nonzero continuous point derivation on $A$
at $a$.
We gave a necessary and sufficient condition
in order that there exist a nonzero continuous
point derivation on $A$ at a given $b\in\bdy U$.
This condition involved a set function known as
lower $(1+\alpha)$-dimensional Hausdorff content,
and denoted $M_*^{1+\alpha}$. The existence
of the derivation is equivalent to the convergence
of a Wiener-type series
$$ \sum_{n=0}^\infty 4^n M_*^{1+\alpha}(A_n(b)\sim U),$$
where $A_n(b)$ denotes the annulus consisting of those $z$ with
$$ \frac1{2^{n+1}} \le |z-b| \le \frac1{2^n}.$$
For the definition of the content we refer the reader to
\cite{LO}.  We shall not use the content in the present
paper, but we note that the above condition is
explicit and practical, 
and allows one to determine by computation whether or not
there exists such a point derivation at a given
boundary point.

The paper \cite{LO} has similar characterizations of the
existence of higher-order continuous point derivations
on $A$.  It also has results about the \lq\lq big Lip" algebra
of all $\Lip\alpha$ holomorphic functions on $U$.
In the latter case, the character space is still 
$Y$, but the results are about weak-star continuous
derivations --- the big Lip algebra always has
nonzero continuous point derivations at every point,
but only the weak-star continuous ones are of any
interest in connection with analytic function theory.

(For the avoidance of confusion, we note that what is
here denoted $A_\alpha(U)$ was denoted $a(U)$ in \cite{LO},
while the notation $A(U)$ was there used for the big Lip
version.)

These results are rather abstract, and the purpose of the 
present paper is to relate them to concrete classical
ideas of derivative.  We are going to confine attention
to the simplest case: the little lip algebra $A$, and
(first-order) continuous point derivations.
The question we address is the following:

\medskip
Suppose the bounded open set $U$, a boundary point $b$,
and $\alpha\in(0,1)$ are given, and suppose there
exists a nonzero continuous point derivation $\partial$ on 
$A=A_\alpha(U)$ at $b$.  Can we evaluate $\partial f$ by a formula
$$ \partial f = c\lim_{n\uparrow\infty}
\frac{f(z_n)-f(b)}{z_n-b},$$
valid for every $f\in A$?   (Here, as before, $f(b)$
denotes the value at $b$ of the unique continuous
extension of $f$ to $\bdy U$.)

\medskip
We remark that a continuous point derivation at $b$
is uniquely determined by its value
$c=\partial z$ at the identity function $z\mapsto z$.
This follows using elementary algebra from the
fact \cite[Lemma 1.1]{LO}
that the set of functions
$f\in A$ that are holomorphic on a neighbourhood of $b$
is a dense subalgebra of $A$.  We say that
$\partial$ is {\em normalised } if $\partial z=1$.
If $\partial$ is any nonzero derivation, then
$\partial/c$ is normalised.

In the interest of further simplicity, we confine attention
to the situation in which the boundary point
is nicely accessible from $U$.  We say that
$U$ {\em has an interior cone }at the boundary point $b$ if 
there is a segment $J$ ending at $b$ and a constant 
$t>0$ such that 
$$\dist(z,\C\sim U)\ge t|z-b|,\ \forall z\in J.$$
We call such a segment $J$ a {\em nontangential ray}
to $b$.

We say that a sequence $(z_n)_n$ of points of $U$
{\em converges non-tangentia\-lly to }$b$, written
$z_n\to_\textup{nt} b$, if there exists a constant
$t>0$ such that 
$$\dist(z_n,\C\sim U)\ge t|z_n-b|,\ \forall z\in J.$$

Obviously, 
if $U$ has an interior cone at $b$, then any sequence
converging to $b$ along a nontangential ray $J$ is converging
nontangentially.  However, the existence of a 
sequence converging nontangentially does not
imply that $b$ lies on the boundary of a single
connected component of $U$.  Without going into details
about Hausdorff content, we remark that for a closed ball
$M_*^\beta(\B(a,r))=r^{\beta}$, that for a line segment $J$,
$M_*^\beta(J)=0$ if $\beta>1$, and also that $M_*^\beta$
is countably subadditive. As a result, it is easy to construct
many examples $U$ in which the complement of $U$ is a countable union
of  closed balls, line segments and the singleton $\{b\}$, and in which
$A$ has a continuous point derivation at $b$.  
All you have to do is make sure that the sum of
the $(1+\alpha)$-th powers of
the radii of all the closed balls that meet
$A_n(b)$ is no greater than $s_n/4^n$, where $\sum_ns_n<+\infty$.

Our main result is the following:

\begin{theorem}
Let $0<\alpha<1$, let $U\subset\C$ be a bounded open set,
let $b\in\bdy U$, $z_n\in U$, $z_n\to_\textup{nt} b$.
Suppose $A=A_\alpha(U)$ admits a nonzero continuous point derivation
at $b$.  Let $\partial$ be the normalised
derivation at $b$. 
Then  for each $f\in A$, we have
$$ \frac{f(z_n)-f(b)}{z_n-b} \to \partial f.$$
\end{theorem}


\begin{corollary}
Under the same hypotheses,
if $U$ has an interior cone at $b$, and $J$ is a nontangential
ray to $b$, then
 $$
\partial f = \lim_{z\to b, z\in J}
\frac{f(z)-f(b)}{z-b}, \ \forall f\in A.$$
\end{corollary}

We would expect that these results could be extended
to higher order derivations and to weak-star
continuous derivations on the big Lip algebra.

In broad outline, the methods we shall use are
adapted from those used to prove a similar result
about bounded analytic functions in \cite{O}. In that paper,
we used duality ideas from functional analysis, the Riesz
Representation Theorem, and the Cauchy transform.
In order to transfer the methods to the Lipschitz algebra
one has to overcome various technical problems.

The methods we use in this paper, using duality and 
abstractions from functional analysis, are termed
nonconstructive, when contrasted with explicit methods
that involve direct use of the Cauchy integral formula
applied to individual functions. It should also be
possible to approach the proof in a constructive spirit.
This would involve the explicit use of Hausdorff
contents, and it is to some extent a matter of taste
(perhaps influenced by familiarity with various
techniques) which might be regarded as preferable.
We leave the constructive approach for another
day.  

As will be seen, we extend here the arsenal of techniques
for dealing with spaces of holomorphic functions in
Lipschitz classes. We expect that these techniques will
prove useful in dealing with other problems in the same area,
such as the behaviour of functions near a special
boundary point which may not be accessible by nontangential
approach, and various approximation problems.. 

Throughout the paper, $0<\alpha<1$, $U\subset\C$ is a bounded
open set, $Y=\clos(U)$, and $b\in X=\bdy U$.

\section{Extensions and Distributions}
\subsection{}
We have already remarked that the elements of $A=A_\alpha(U)$ 
extend uniquely to continuous functions on $Y=\clos(U)$. It is
in fact obvious that this extension imbeds $A$ isometrically
as a closed subalgebra of $\lip(\alpha,Y)$. 

We are interested in point derivations, and point derivations
annihilate the constant functions (because
$\partial1= \partial(1^2) = 2\cdot1\cdot\partial1$),
so we are really more interested in the quotient
space $A/\C$ modulo constants, and the \lq pure'
Lipschitz seminorm $\|f\|'=\kappa(f)$ from \eqref{E:1}.
We note that the extension to $Y$ also preserves
this pure seminorm.

Our first significant point is that the restriction
of the extension to the boundary $X$ is also isometric
on both the norm and the pure seminorm.
That the sup norm is preserved is the classical
maximum principle.  But we also have:

\begin{lemma}
Let $f\in\Lip(\alpha,Y)$ be holomorphic on $U$.
Then
$$
\sup\left\{
\frac{|f(z)-f(w)|}{|z-w|^\alpha}:
z,w\in Y
\right\}
=
\sup\left\{
\frac{|f(z)-f(w)|}{|z-w|^\alpha}:
z,w\in X
\right\}
$$
\end{lemma} 
 
\begin{proof}
For each fixed $w_0\in Y$, 
Let $\Sigma(w_0)$ be the Riemann surface
of $\log(z-w_0)$, and let $p:\Sigma\to\C\sim\{w_0\}$
be the covering map, so that $\log(Z-w_0)$
is a well-defined holomorphic function 
on $\Sigma(w_0)$, with real part
$\log|p(Z)-w_0|$ and imaginary part
$\arg(Z-w_0)$.
Then $(Z-w_0)^\alpha$, interpreted as 
$$\exp(\alpha\log|p(Z)-w_0|+i\alpha\arg(Z-w_0)),$$
is also holomorphic on $\Sigma(w_0)$, and
$$
g(Z):=\frac{f(p(Z))-f(w_0)}{(Z-w_0)^\alpha}
$$
is a well-defined function on $\Sigma(w_0)\cap p^{-1}(Y)$, 
with absolute value that depends only on 
the projection $p(Z)\in\C$.  It is holomorphic on
$p^{-1}(U)$.  

Since 
$\displaystyle\frac{|f(z)-f(w)|}{|z-w|^\alpha}$
is continuous on $Y\times Y$, it attains its supremum,
say $m$,
at some point $(z_0,w_0)\in Y\times Y$.  If
$z_0\not\in X$, then we have a contradiction to the maximum
principle unless $|f(p(Z))-f(w_0)|$ is identically
equal to $m|p(Z)-w_0|^\alpha$ on the connected
component of each preimage of $z_0$ in $p^{-1}(U)$, and
hence at some point of $p^{-1}(X)$. Thus
we may assume $z_0\in X$.  Similarly, we may assume
$w_0\in X$. 
\end{proof}

As a result, we may regard $A/\C$ as a subspace
of $\lip(\alpha,X)/\C$.   Thus, by the Hahn-Banach Theorem,
each continuous linear functional $T\in A^*$
that annihilates the constants has an isometric
extension to the whole of $\lip(\alpha,X)/\C$,
hence (by a standard method) may be
represented by a Borel-regular measure
$\mu$ on $X\times X$ having no mass on the diagonal
and total mass equal to $\|T\|'$ (the dual
norm on $A^*$ to $\|\dot\|'$), via a formula
$$ Tf = \int_{X\times X} 
\frac{f(z)-f(w)}{(|z-w|^\alpha}
\,d\mu(z,w),\ \forall f\in A.$$

In particular, if we assume that $\partial$
is a continuous point derivation on $A$,
then it has a representation of this kind.

\subsection{Extensions}
Let $\lip\alpha$ denote, for short, the global space
$\lip(\alpha,\C)$ of bounded $\lip\alpha$
functions on $\C$.  Each $f\in A$ may be extended
(in many different ways) to an element of 
$\lip\alpha$, without increasing its pure
norm $\|f\|'$ or supremum. (In fact it is
not difficult to check that if $\omega(r)$
is any concave upper envelope for the modulus
of continuity of $f$, then $f$ may be extended
so that its modulus of continuity
remains bounded by $\omega$. For instance, 
this may be seen by applying the method used
to prove Kirszbraun's Theorem
in \cite{Fed}.) 
Thus the restriction
map to $U$ (or $Y$, or $X$) makes $A$ isometric
to a quotient space of 
$$ \tilde A = \{ f\in \lip\alpha: 
f \textup{ is holomorphic on }U\}.$$
We shall find it convenient to work
with globally-defined functions in the sequel.

\subsection{Distributions}
Let $\mathcal D$ denote the space of test functions
(i.e. $C^\infty$ functions having compact support),
and let $\mathcal D'$ denote its dual, the Schwartz
distribution space.
If 
$\mu$ is any complex measure
on $X\times X$, having no mass on the diagonal,
then we may define
a distribution $T_1\in\mathcal D'$ by setting
\begin{equation}\label{E:2}
\langle\phi,T_1\rangle =
 \int_{X\times X} 
\frac{\phi(z)-\phi(w)}{|z-w|^\alpha}
\,d\mu(z,w),\ \forall \phi\in \mathcal D.
\end{equation}
This distribution $T_1$ will not, in general,
be representable by integration against 
a locally-integrable funtion
or a measure, but will extend continuously to
an element of $(\lip\alpha)^*$. It is a bit
\lq wilder' than a measure.  (It may be represented
in the form 
$$ T_1 = \nu_0 + \frac{\partial}{\partial x}\nu_1
+ \frac{\partial}{\partial y}\nu_2,$$
where the $\nu_j$ are compactly-supported measures,
but we shall not use this representation, as it 
carries less information than the fact that
$T_1$ acts continuously on $\lip\alpha$.)
We denote the extension of $T_1$
to $\lip\alpha$ by the same notation $T_1$,
and write its value as $T_1(f)$ or
$\langle f,T_1\rangle$ for any 
$f\in\lip\alpha$.

Since $\langle\phi,T_1\rangle$ is unaffected if $\phi$ is
altered away from $X$, it is clear that
$T_1$ has support in $X$. Thus we can also define
$\langle\phi,T_1\rangle$ for any function $\phi$
defined and $C^\infty$ on a neighbourhood of $X$
to be 
$\langle\tilde\phi,T_1\rangle$ where $\tilde\phi$
is any element of $\mathcal E$ (the space of
globally-defined $C^\infty$ functions) that agrees
with $\phi$ near $X$.   For instance, 
$\displaystyle \left\langle \frac1{z-a},T_1\right\rangle $
makes sense, for $a\not\in X$. Similarly, $\langle f,T_1\rangle$
makes sense whenever $f$ is defined on some neighbourhood
of $X$ and satisfies a little-lip$\alpha$ condition there. 

\subsection{Cauchy Transforms} 
The main idea behind what follows is that
although $T_1$ is wilder than a measure, 
it is still tame enough to allow
us to treat it almost as though it were
a measure.  
Specifically,
the Cauchy transform of $T_1$ (which we are about to define)
is
representable by integration against a locally-integrable 
function.  This fact was already 
noted and exploited in \cite{Annih}.

The Cauchy transform of $\phi\in\mathcal D$
is its convolution
$$ \hat\phi := \phi*\left(
\frac1{\pi z}
\right) $$
with the fundamental solution of $\displaystyle 
\frac{\partial}{\partial\bar z}$. In other words,
$$ \hat\phi(z) = \frac1\pi \int\frac{\phi(\zeta)}{z-\zeta} dm(\zeta),$$
for all $z\in\C$, where $m$ denotes area measure.
This function belongs to the space $\mathcal E$,
and satisfies
$$ \frac{\partial \hat\phi}{\partial z} = \phi. $$

For distributions $T$ having compact support, we define
$$ \langle\phi,\hat T\rangle =
-\langle\hat\phi, T\rangle, \forall\phi\in\mathcal D.$$

If $T_1$ is given by Equation \eqref{E:2}, then
consider
$$ H(a) = \frac1\pi \int
\frac{z-w}{(z-a)(w-a)|z-w|^\alpha}
\,d\mu(z,w),$$
for $a\in \C$.  This is well-defined whenever
$$ \tilde H(a) = \int
\frac{|w-z|^{1-\alpha}}{|z-a|\cdot|w-a|}
\,d|\mu|(z,w)<\infty,$$
which happens almost everywhere with respect to area measure,
and $\tilde H$ is locally-integrable,
as is seen by an application of Fubini's Theorem.
Also $|H(a)|\le \tilde H(a)$ for all such $a$.
Another Fubini calculation yields
$$ \langle\phi, \widehat{T_1}\rangle = \int_\C \phi\cdot H dm, $$
for all $\phi\in\mathcal D$.  Thus
$H$ represents $\widehat{T_1}$.  Based on this, 
we sometimes write $\widehat{T_1}(a)$ for $H(a)$.
Note that
$$ \widehat{T_1}(a) = H(a)=\left\langle\frac1{\pi(a-z)},T_1\right\rangle,$$
whenever $a\not\in X$.

Note that these facts do not depend on the relation
of $T_1$ to a derivation, but only on its
representability in the form \eqref{E:2} for some
measure $\mu$ on $X\times X$.

Note also, for future reference, that if $T$
actually represents a normalized point derivation at a point $b$ of $X$, 
then
 $\displaystyle H(a) = \frac{1}{\pi(b-a)^2}= \oh(1/|a|^2)$
as $a\to\infty$.

\section{Estimates}
\subsection{The product $g\cdot T_1$}
The dual of any Banach algebra is naturally
a module over the algebra.  In the present
situation, $\lip\alpha$ acts on $(\lip\alpha)^*$,
so given $g\in\lip\alpha$ and $T_1$ as in Equation \eqref{E:2},
we may define a new element $g\cdot T_1$ of 
$(\lip\alpha)^*$ by setting
$$ \langle\phi,g\cdot T_1\rangle = \langle g\cdot \phi,T_1\rangle,
\ \forall\phi\in\mathcal D.$$
We remark that $\langle1,g\cdot T_1\rangle = \langle g,T_1\rangle\not=0$,
in general, so we cannot represent $gT_1$ by a measure as in
Equation \eqref{E:2}.  However, writing
$$ \phi(z)g(z)-\phi(w)g(w)=
\left( \phi(z)-\phi(w)\right)\cdot g(z) 
+ \phi(w)\cdot\left( g(z)-g(w)\right),$$
a short calculation gives
$$ \langle\phi,g\cdot T_1\rangle
= \int_{X\times X}
\frac{\phi(z)-\phi(w)}{|z-w|^\alpha}
d\mu'(z,w)
+
\int_X\phi(w)\,d\lambda(w),$$
where $\mu'$ is the measure on $X\times X$ such that
$$ \mu'(E) = \int_E g(z)d\mu(z,w) $$
whenever $E\subset X\times X$ is a Borel set,
and $\lambda$ is the measure on $X$ such that
$$ \lambda(E) = \int_{E\times X}
\frac{g(z)-g(w)}{|z-w|^\alpha}d\mu(z,w)
$$
whenever $E\subset X$ is Borel, i.e. $\lambda$
is the first-coordinate marginal of the measure
$$  
\frac{g(z)-g(w)}{|z-w|^\alpha}\cdot\mu(z,w)
$$
(a {\em bounded} multiple of $\mu$).
So we may write $g\cdot T_1 = S_1+S_2$, where
$S_1\perp\C1$ is represented (as in Equation \eqref{E:2})
by the measure $\mu' = g(z)\cdot\mu$ on $X\times X$, and
$S_2$ is represented by the measure $\lambda$ on $X$.   

Denoting the total variation of a measure $\mu$ by $\|\mu\|$,
we note for future reference that
$$ \|\lambda\|\le \kappa(g)\cdot\|\mu\|.$$

Let us call $S_1$ {\em the main part of } $g\cdot T_1$
and $S_2$ {\em the residual part of } $g\cdot T_1$.

\subsection{Estimate}
We are aiming for an estimate for the growth
of the Cauchy transform of $g\cdot T_1$
as we approach a boundary point nontangentially.
The main step is an estimate for the Cauchy
transform 
$\hat{S}_1$ of the main part.

\begin{lemma}\label{L:1} Fix a measure $\mu$ on $X\times X$.
Let $b\in X$ and let
$g\in \lip\alpha$ have $g(b)=0$.
Let $S_1$ be the distribution
given by
$$ \langle\phi, S_1\rangle
=
\int_{X\times X}
\frac{\phi(z)-\phi(w)}{|z-w|^\alpha}g(z)d\mu(z,w),\ \forall\phi\in
\mathcal D.$$
Fix $t$ with $0<t<1$. Then there is a constant
$c$ that depends only on $t$ such that 
$$|\hat S_1(a)|\le \frac{c\kappa(g)\cdot\|\mu\|}{
|a-b|} $$
for all $a\in\C\sim X$ with $\dist(a,X)\ge t|a-b|$.
\end{lemma}

\begin{proof}
We may assume without loss in generality that $\kappa(g)=1$.
Then $|g(z)| \le |z-b|^\alpha$ for all $z\in\C$.

Assume $\dist(a,X)\ge t|a-b|$.  

Let $z\in X$. If $|z-b|\le 2|a-b|$, then 
$$ |z-b|\le \frac{2\dist(a,X)}t\le \frac{2|z-a|}t, $$
whereas if $|z-b|>2|a-b|$, then
$$ |z-b| \le |z-a|+|a-b| < |z-a|+\frac12|z-b|,$$
so $|z-b|<2|z-a|$. Thus in either case
$|z-b|\le 2|z-a|/t$. Hence
$|g(z)|\le c|z-a|^\alpha$ for all $z\in X$, where $c$ depends only on
$t$.  Henceforth we shall use $c$ to denote a
constant, which may differ at each occurrence,
depending only on $t$.
 
Let 
$$ K(z,w) = 
\frac{|w-z|^{1-\alpha}}{|z-a|\cdot|w-a|}.
$$
Then
$$|\hat S_1(a)|\le \int_{X\times X} K(z,w) |g(z)| d|\mu|(z,w).$$

We have $|z-w|^{1-\alpha}\le |z-a|^{1-\alpha}+|w-a|^{1-\alpha}$,
so for $z,w\in X$ we have
$$ K(z,w)  \le
\frac1{|z-a|\cdot|w-a|^\alpha}
+
\frac1{|z-a|^\alpha\cdot|w-a|}
,$$
so
$$ K(z,w)|g(z)|  \le
\frac c{|z-a|^{1-\alpha}\cdot|w-a|^\alpha}
+
\frac c{|w-a|}
\le \frac c{\dist(a,X)}.$$
The 
 desired result follows.
\end{proof}

\begin{lemma}\label{L:2}
If $T_1$ is given by Equation \eqref{E:2}, and $0<t<1$,
then there is a constant $c$ depending only on $t$
such that 
$$ |\widehat{g\cdot T_1}(a)| 
\le\frac{c\kappa(g)\|\mu\|}{\dist(a,X)}$$
whenever $\dist(a,X)\ge t|a-b|$.
\end{lemma}
\begin{proof}
Let $S_1$ and $S_2$ be the parts of $g\cdot T_1$ and
$\lambda$ represent $S_2$, as in the last section. 
Since $\lambda$ is a measure supported on $X$, we have
$$|\hat\lambda(a)|\le \frac{\|\lambda\|}{\dist(a,X)}
\le \frac{\|\mu\|\cdot\kappa(g)}{\dist(a,X)}$$
whenever $\dist(a,X)\ge t|a-b|$.
Combining this with the last lemma, we get
$$ |\widehat{g\cdot T_1}(a)| 
\le |\hat S_1(a)| + |\hat S_2(a)|
\le\frac{c\kappa(g)\|\mu\|}{\dist(a,X)},$$
as required.
\end{proof}

\subsection{Estimate for $\hat T_1$}
By a similar (slightly simpler) argument we obtain the following.
\begin{lemma}\label{L:3}
Let $T_1$ be given by Equation \eqref{E:2}. Then for each $t$ with $0<t<1$
there exists a constant $c$, depending only on $t$, such that
$$ |\hat{T_1}(a)| 
\le\frac{c\|\mu\|}{\dist(a,X)^{1+\alpha}},$$
whenever $\dist(a,X)\ge t|a-b|$.
\end{lemma}
\section{Proof of Theorem}

Suppose $A=A_\alpha(U)$ admits a nonzero
continuous point derivation at $b$, and let
$\partial$ be the normalised derivation at that point.
Given a function $f\in A$, we use the same symbol
$f$ to denote some global extension in $\lip\alpha$.
Note that the extension is uniquely-determined on $X$,
but not outside $Y$. None of the quantities we will
consider depend on which extension is taken.

As we have seen, there is a measure $\mu$ on $X\times X$,
having no mass on the diagonal, such that
$$ \partial f = 
\int_{X\times X} 
\frac{f(z)-f(w)}{(|z-w|^\alpha}
\,d\mu(z,w),\ \forall f\in A.$$
Let $T_1$ be the distribution defined by
Equation \eqref{E:2}. Then $T_1$ has support in $X$,
and extends continuously to an element of 
$\lip\alpha^*$.

Let $\mathcal A$ denote the set of
all $f\in\lip\alpha$ that are holomorphic
on $U$ and on a neighbourhood of $b$.
As remarked earlier, $\mathcal A$ is
dense in $A$.

$\mathcal E'$ is a module over $\mathcal E$, via the 
multiplication operation $(\lambda,T)\mapsto
\lambda\cdot T$
defined by 
$$ \langle\phi,\lambda\cdot T\rangle
= \langle\lambda\phi,T\rangle $$
for all $\phi,\lambda\in\mathcal E$
and $T\in\mathcal E'$. So we may define
$T_0=(z-b)\cdot T_1$, where
$(z-b)$ denotes the function 
$z\mapsto(z-b)$.

We calculate that
for $f\in\mathcal A$ we have
$$\langle f,T_0\rangle
=
\langle(z-b)f,T_1\rangle = f(b),
$$
and hence by continuity this also holds
for all $f\in A$, i.e.
the distribution $T_0$ represents
evaluation at $b$ on $A$.
\def\dbar{\frac{\partial}{\partial\bar z}}
Next,
$$ \dbar\left((z-b)\cdot \hat T_1\right)
= (z-b)\cdot T_1 = T_0 = \dbar \hat T_0,$$
so by Weyl's Lemma (cf. \cite[p.72]{Schw} or
\cite[Theorem 4.4.1, p. 110]{HORM})
$$ \hat T_0 = (z-b)\cdot \hat T_1 + h,$$
where $h$ is an entire function.
Also, if $\phi\in\mathcal D$ vanishes on a
neighbourhood of $Y$,
then
$$ 
\langle\phi,\hat T_0\rangle
=\hat\phi(b) =
\langle\phi,\frac1{\pi(b-z)}\rangle.$$
Thus $\displaystyle
\hat T_0(z) = \frac1{\pi(b-z)}$ off $Y$.
In particular, $\hat T_0(z)$ tends to $0$
as $z\to\infty$.  Now we also have
$$ \hat T_1(z) = \Oh\left(
\frac1{|z|^2}\right) $$
as $z\to\infty$, hence $h(z)\to0$ as $z\to\infty$,
whence $h$ is identically zero, and
$$ \hat T_0 = (z-b)\hat T_1.$$

Next, define $T=-\pi(z-b)\cdot T_0=-\pi(z-b)^2\cdot T_1$. Then 
$T$ annihilates $A$, and by a similar argument to that
above we see that
$$ \hat T = -\pi(z-b)^2\hat T_1 + k$$
for some entire $k$. Now the fact that
$T$ annihilates $A$ forces $\hat T=0$ off $Y$,
so we get
$$ \hat T = 1-\pi(z-b)^2\cdot \hat T_1.$$

Now suppose $(z_n)_n\subset U$ and
$z_n\to_{\textup{nt}}b$.
By Lemma \ref{L:3}, 
\begin{equation}\label{E:T-hat-1}
\hat T(z_n) - 1 = -\pi(z_n-b)^2\cdot\hat T_1(z_n) 
= \Oh(|z_n-b|^{1-\alpha})\to 0
\end{equation}
as $n\uparrow\infty$.
In particular, $\hat T(z_n)\to1$,
so we may choose $N_1\in\N$
such that $|\hat T(z_n)|>\frac12$ for $n>N_1$.

Consider any point $a\in U$ with $\hat T(a)\not=0$, and define
$$ R_a = \frac1{\pi\hat T(a)(a-z)}\cdot T
=\frac{(z-b)^2}{\hat T(a)(z-a)}\cdot T_1,$$
i.e.
$$ \langle\phi,R_a\rangle = \frac1{\hat T(a)}\left\langle
\frac{\phi(z)}{\pi(a-z)}, T\right\rangle,\ \forall\phi\in\mathcal D.$$
This is a well-defined distribution since
$\phi(z)/(z-a)$ is $C^\infty$ near $X$, and hence near
the support of $T$. Also $R_a$ is supported on $X$,
and represents $a$ on $A$, since for $f\in\mathcal A$
we may write $f(z)=f(a)+(z-a)g(z)$ for a $g\in\mathcal A$,
and get
$$
\langle f,R_a\rangle =
\frac1{\hat T(a)}\left\langle
\frac{f(a)}{\pi(a-z)}+\frac{g(z)}{\pi}, T\right\rangle
= f(a) - \frac{\langle g,T\rangle}{\pi\hat T(a)} = f(a).$$
Thus the functional
$$ f\mapsto \frac{f(a)-f(b)}{a-b} - \partial f $$
is represented on $A$ by the distribution
$$
D_a = \frac{R_a-T_0}{a-b} - T_1.
$$ 
Hence $D_{z_n}(f) \to 0$ as $n\uparrow\infty$,
for all $f\in \mathcal A$.
To prove the theorem, we have to show that
this also holds for all $f$ in the closure $A$ of $\mathcal A$.
To do this, it suffices to show that
the functionals $D_{z_n}$ are uniformly
bounded on $A$, for  $n\ge N_1$, i.e
that
$$ |D_{z_n}(f)|  \le c\kappa(f) $$
for some constant $c>0$, for all $f\in A$ and all $n>N_1$.

Fix an arbitrary $f\in\lip\alpha$, holomorphic on $U$.
Take $g(z)=f(z)-g(b)$, so $D_a(f)=D_a(g)$, $\kappa(f)=\kappa(g)$ and 
$g(b)=0$.  

Let 
$$\displaystyle L_a = \hat T(a)R_a = \frac{(z-b)^2}{z-a}\cdot T_1 =
\frac1{\pi(a-z)}\cdot T.$$ 
Then
$D_a = E_a+F_a$, where
$$\begin{array}{rcl}
\displaystyle E_a &=& \displaystyle \frac{L_a-T_0}{a-b} - T_1,\\
F_a &=& \displaystyle \frac{R_a-L_a}{a-b}.
\end{array}
$$.

For $a\in U$, we calculate
$$ L_a-T_0 = -
\frac{T}{\pi(z-a)}+
\frac{T}{\pi(z-b)}
= \frac{b-a}{\pi(z-a)(z-b)}\cdot T,
$$
$$ \frac{L_a-T_0}{a-b} = -\frac{T}{\pi(z-a)(z-b)}=\left(\frac{z-b}{z-a}\right)\cdot T_1,
$$
$$ E_a = \left(\frac{a-b}{z-a}\right)\cdot T_1,
$$
so
$$
 E_a(g)
= (a-b)\cdot\left\langle\frac{g(z)}{z-a},T_1
\right\rangle = (a-b)\cdot \widehat{g\cdot T_1}(a).$$
Thus Lemma \ref{L:2} gives
$$| E_{z_n} g | \le |z_n-b| \cdot \frac{c\kappa(g)\cdot \|\mu\|}{\dist(z_n,X)}
\le c\kappa(f),$$

Next,
$$ R_a-L_a = \left(1-\hat T(a)\right)R_a.$$
Since $g\in A$ and $g(b)=0$, we have
$$ |R_a(g)| = | g(a) | \le \kappa(g)|b-a|^\alpha.$$
Then for $a=z_n$ with $n\ge N_1$, we get
$$ |(R_a-L_a)(g)| \le c|a-b|^{1-\alpha}\kappa(g)|a-b|^\alpha
=c|a-b|,$$
so
$$  |F_{z_n}(g)| \le c, \textup{ for }n\ge N_1.  $$
Thus $D_{z_n}(g)$ is indeed bounded, as required.
This concludes the proof.

\medskip\noindent\textbf{Note:} This article was originally published in 
Analysis and Mathematical Physics 4 (2014) 131-44. 
DOI: 10.1007/s13324-014-0074-0. \copyright Copyright 2014 Springer-Verlag. 
The present ArXiv version corrects three misprints, one on p.12 and
two on p.14.  These did not materially affect the validity of the argument,
and the result stands.  
\end{document}